\documentclass{amsart}

\usepackage{graphicx}

\newtheorem{theorem}{Theorem}[section]
\newtheorem{lemma}[theorem]{Lemma}
\newtheorem{proposition}{Proposition}[section]
\theoremstyle{definition}
\newtheorem{definition}[theorem]{Definition}

\newtheorem{corollary}{Corollary}[section]
\newtheorem{conjecture}{Conjecture}[section]
\theoremstyle{remark}
\newtheorem{remark}[theorem]{Remark}

\numberwithin{equation}{section}

\begin{document}

\title{The prime pairs are equidistributed among the coset lattice congruence classes}

\author{T. Agama *}
\address{Department of Mathematics, African Institute for Mathematical science, Ghana
}
\email{theophilus@aims.edu.gh/emperordagama@yahoo.com}

\author{M. Bortolamasim **}
\address{ Engineer and Mathematician, Ordine degli Ingegneri della Provincia di Modena, c/o Dipartimento di Ingegneria E. Ferrari, Università di Modena e Reggio Emilia, Via P. Vivarelli 10, 41125 MODENA (Italy)
}
\email{bortolamasim@libero.it.}

\author{A. Tapia ***}
\address{ Ph.D. in Physics, Czech Technical University (Czech Republic)
}
\email{ aortiztapia2013@gmail.com.}


\date{\today}



\begin{abstract}
In this paper we show that for some constant $c>0$ and for any $A>0$  there exist some $x(A)>0$  such that, If $q\leq (\log x)^{A}$ then we have \begin{align}\Psi_z(x;\mathcal{N}_q(a,b),q)&=\frac{\Theta (z)}{2\phi(q)}x+O\bigg(\frac{x}{e^{c\sqrt{\log x}}}\bigg)\nonumber
\end{align}for $x\geq x(A)$ for some $\Theta(z)>0$. In particular for $q\leq (\log x)^{A}$ for any $A>0$\begin{align}\Psi_z(x;\mathcal{N}_q(a,b),q)\sim \frac{x\mathcal{D}(z)}{2\phi(q)}\nonumber
\end{align}for some constant $\mathcal{D}(z)>0$ and where $\phi(q)=\# \{(a,b):(p_i,p_{i+z})\in \mathcal{N}_q(a,b)\}$. 
\end{abstract}

\maketitle

\section{\textbf{Introduction}}
It is known that there are infinitely many primes in arithmetic progression \cite{hildebrand2005introduction}. That is to say, the quantity \begin{align}\lim \limits_{x\longrightarrow \infty}\# \{p\leq x:p\equiv a\pmod q,~(a,q)=1\}=\infty.\nonumber
\end{align}This is now known as Dirichlet Theorem of primes in arithmetic progression. Even the more potentially informative question arises concerning their distribution among the primitive congruence classes. Again given a fixed $q$ allowed to grow slowly we can claim equidistribution of the primes among the primitive congruence classes. In particular the well-known Siegel-Walfisz theorem states (See \cite{May})\begin{theorem}Let $c>0$ be some constant. For any $A>0$ there exist some $x(A)>0$ such that, If $q\leq (\log x)^{A}$ then \begin{align}\Psi(x;q,a)=\frac{1}{\phi(q)}x+O\bigg(\frac{x}{e^{c\sqrt{\log x}}}\bigg)\nonumber
\end{align} for $x\geq x(A)$, where $\Psi(x;q,a)=\# \{p\leq x: p\equiv a\pmod q,~(a,q)=1\}$
\end{theorem}In this paper by adapting some methods employed in establishing this result plus some new ideas, we show the following \begin{theorem}For some constant $c>0$ and for any $A>0$,  there exist some $x(A)>0$  such that If $q\leq (\log x)^{A}$ then \begin{align}\Psi_z(x;\mathcal{N}_q(a,b),q)&=\frac{\Theta (z)}{2\phi(q)}x+O\bigg(\frac{x}{e^{c\sqrt{\log x}}}\bigg)\nonumber
\end{align}for $x\geq x(A)$. In particular for $q\leq (\log x)^{A}$ for any $A>0$\begin{align}\Psi_z(x;\mathcal{N}_q(a,b),q)\sim \frac{x\mathcal{D}(z)}{2\phi(q)}\nonumber
\end{align}for some constant $\mathcal{D}(z)>0$ and where $\phi(q)=\# \{(a,b):(p_i,p_{i+z})\in \mathcal{N}_q(a,b)\}$. 
\end{theorem}This result establishes equidistribution among the coset lattice congruence classes. 

\section{\textbf{Notations}}
In this paper the quantity $\pi_{z}(x; q,\mathcal{N}(a,b)):=\# \{(p_i,p_{i+z}):(p_i,p_{i+z})\in \mathcal{N}_q(a,b), \quad \nonumber \\ p_i,p_{i+z}\leq x\}$. Also the function $\omega(q)$ will be referred to as $\omega(q)=\# \{(a,b):(n_1,n_2)\equiv (a,b)\pmod q\}$. As is standard, the inequality $|f(n)|\leq Kg(n)$ for some constant $K>0$ for some $n\geq n_0$ will be compactly written as $f(n)\ll g(n)$. Similarly, the inequality $|f(n)|\geq Mg(n)$ for some $M>0$ for all $n\geq n_0$ will often be written as $f(n)\gg g(n)$.  In situations where both inequalities hold, then we will write $f(n)\asymp g(n)$. In the case where the implicit constant depends on another parameter, say $z$, then we write $f(n)\asymp_z g(n)$. As is customary, $\epsilon$ and $\delta$ are positive numbers that are usually taken to be small. In cases where they depend on some variable $n$, then we will write $\delta(n)$ and $\epsilon(n)$. The limit $\lim \limits_{n\longrightarrow \infty}\frac{f(n)}{g(n)}=0$ will be compactly written as $f(n)=o(g(n))$ and for $\lim \limits_{n\longrightarrow \infty}\frac{f(n)}{g(n)}=1$, we will write $f(n)\sim g(n)$. We will instead denote the greatest common divisor of, say $a$ and $b$, as $\gcd(a,b)$. This is to avoid interference with the integer lattice notation. Also through out this paper we will choose to work with coset lattice congruence classes. We will denote this coset lattice congruence class modulo $q$ with representative $(a,b)$ by the letter $\mathcal{N}_q(a,b)$. The problem is much more tractable within this framework. Also we will work within the half Cartesian plane when handling pairs of positive integers rather than the entire Cartesian plane.

\section{\textbf{Overview and Idea of proof of main result}}
The main result of this paper can be thought of as a solution to a problem in the Cartesian plane. The quest is to establish an analogue of the classical Dirichlet theorem of primes in arithmetic progression. The steps that goes into addressing this sort of problem are outline vividly in the following sequel. 

\begin{enumerate}
\item [$\bullet$] First we construct and label containers, which in our language we choose to call cosets, and carry them along into the lower half of the positive Cartesian plane. Each of the labeled containers admits only equivalent  pairs. That is pairs that are not equivalent cannot possibly live in the same container. 
\bigskip

\item [$\bullet$] For each of these labeled containers we deposit equivalent prime pairs. It is important to remark there will not be any empty container since there is a one-to-one correspondence between these containers and primitive  congruence classes in any given modulus, $q$ say.
\bigskip

\item [$\bullet$] In each container we designate a representative. The representatives are those pairs of the form $(a,b)\pmod q$ with $a,b<q$. For each of these representatives we squash them onto the positive horizontal axis of the plane. This then converts the problem to understanding pair correlation of an arbitrary gap.
\bigskip

\item [$\bullet$] Next by applying the orthogonality principle employed in the proof of the Classical Dirichlet theorem with the aid of the newly introduced character, which is a two dimensional version of the Dirichlet character, we carry out the standard multiplicative Fourier expansion.
\bigskip

\item [$\bullet$] Next we appeal to the area method to decompose our correlation to a bi-linear sum. This leaves us with a sum that is well understood and was employed in the literature to establish equidistribution on the primes. By an application of partial summation to this bi-linear sum and plugging into the main sum, the result then follows. This immediately establishes equidistribution of prime pairs of an idealized gap.   

\end{enumerate}

\section{\textbf{Preliminary results}}
In this section we state some results that will partly be needed for establishing the main result of this paper. The results in this section form the foundation for establishing our main result of this paper.
\bigskip

\begin{lemma}\label{starter}Let $\Pi_x:=\{(p_1,p_2):p_1\leq x,~p_2\leq x\}$
then \begin{align}\# \Pi_x=\frac{\pi(x)(\pi(x)-1)}{2}\nonumber
\end{align}where $\pi(x)$ is the prime counting function.
\end{lemma}

\begin{proof}
Consider the positive Cartesian plane and slice it into two so that the lower half contains finitely many points. Let $\Pi_x$ denotes the set of all prime pairs in this finite region. By letting $M$ and $N$ be the horizontal and the vertical boundary of the lower region, we set $|M|=|N|=x$ and consider the primes along the horizontal and the vertical boundary of the bounded region given by  $M_x=\{p\leq x: p\in \rho\}$ and $N_x=\{p\leq x:p\in \rho\}$. More explicitly we consider the sets \begin{align}M_x=\{p_1,p_2,\ldots, p_{\pi(x)}\}\nonumber
\end{align}and \begin{align}N_x=\{p_1,p_2,\ldots,p_{\pi(x)}\}.\nonumber
\end{align}Now we construct the set \begin{align}\mathcal{C}=\bigcup \limits_{j=1}^{\pi(x)-1}\{p_j\cdot p_{j+1},\ldots, p_{j}\cdot p_{\pi(x)}\}.\nonumber
\end{align}It follows that the number of distinct prime pairs in the bounded region is given by \begin{align}\# \mathcal{C} =\bigg|\bigcup \limits_{j=1}^{\pi(x)-1}\{p_j\cdot p_{j+1},\ldots, p_{j}\cdot p_{\pi(x)}\}\bigg|.\nonumber
\end{align}Since $\{p_i\cdot p_{i+1},\ldots,p_i\cdot p_{\pi(x)}\}\cap \{p_j\cdot p_{j+1},\ldots,p_{j}\cdot p_{\pi(x)}\}=\emptyset$ for $i\neq j$, it follows that \begin{align}\# \mathcal{C}&=\sum \limits_{j=1}^{\pi(x)-1}\# \{p_j\cdot p_{j+1},\ldots, p_j\cdot p_{\pi(x)}\}\nonumber \\&=\frac{\pi(x)(\pi(x)-1)}{2}.\nonumber
\end{align}This proves the Lemma.
\end{proof}

\begin{remark}
Knowing the number of prime pairs in any finite region of the Cartesian plane is of critical importance. For It sets a score in verifying that certain images weighted by their multiplicity coincides. Next we launch a terminology concerning the integer lattice.
\end{remark}

\begin{definition}
Let $(n_1,n_2)$ and $(m_1,m_2)$ be any two integer lattice. Then we say $(n_1,n_2)$ and $(m_1,m_2)$ are equivalent pairs, denoted $(n_1,n_2)\simeq (m_1,m_2)$ if and only if \begin{align}n_1\cdot n_2\equiv m_1\cdot m_2\pmod q.\nonumber
\end{align}We say the pair $(n_1,n_2)\in (a,b)\pmod q$ if $n_1\equiv a\pmod q$ and $n_2\equiv b \pmod q$.
\end{definition}

\begin{definition}
Let $\gcd(a,q)=\gcd(b,q)=1$ with $a,b<q$. By the coset lattice congruence class modulo $q$ with representative $(a,b)\pmod q$, we mean the set \begin{align}\mathcal{N}=\{(n_1,n_2)\pmod q:(n_1,n_2)\simeq (a,b),~\gcd(n_1,q)=\gcd(n_2,q)=1\}.\nonumber
\end{align}
\end{definition}

\begin{remark}
Next we show that there does exist a correspondence between the coset lattice congruence modulo $q$ with congruence classes modulo $q$.
\end{remark}

\begin{proposition}\label{coset}
Let $\mathcal{M}=\{\mathcal{N}_i\}$ be the set of all coset lattice congruence classes modulo $q$, then \begin{align}\sum \limits_{\mathcal{M}}1=\phi(q).\nonumber
\end{align}
\end{proposition}

\begin{proof}
Consider the map $\mathcal{F}:\mathcal{M}=\{\mathcal{N}_i\}\longrightarrow \{ab\pmod q:\gcd(ab,q)=1\}$ for $\mathcal{F}(\mathcal{N}_i)=ab\pmod q$ if $(a,b)\in \mathcal{N}_i$. It is easy to see that the map is well defined. For suppose that for $\mathcal{N}_i,\mathcal{N}_j\in \mathcal{M}$ such that $\mathcal{N}_i\neq \mathcal{N}_j$ and that $\mathcal{F}(\mathcal{N}_i)=\mathcal{F}(\mathcal{N}_j)$, then it follows that $n_in_{i+1}\equiv n_jn_{j+1}\pmod q$ for any $(n_{i},n_{i+1})\in \mathcal{N}_i$ and $(n_j,n_{j+1})\in \mathcal{N}_j$. This certainly implies that $(n_i,n_{i+1})\simeq (n_j,n_{j+1})$ and it follows that $\mathcal{N}_i=\mathcal{N}_j$. This is a contradiction. Thus the map is well defined. Next we see that the map is injective. For suppose $\mathcal{F}(\mathcal{N}_i)=\mathcal{F}(\mathcal{N}_j)$, then it follows that $n_in_{i+1}\equiv n_{j}n_{j+1}\pmod q$ for any $(n_i,n_{i+1})\in \mathcal{N}_i$ and $(n_j,n_{j+1})\in \mathcal{N}_j$. It follows that $(n_i,n_{i+1})\simeq (n_j,n_{j+1})$ and it must certainly be that $\mathcal{N}_i=\mathcal{N}_j$. This establishes injectivity. Surjectivity is obvious by virtue of definition of the map. Thus the map is in one to one correspondence with the $\phi(q)$ primitive congruence classes modulo $q$. This proves the proposition.
\end{proof}

\begin{remark}
Since any prime pairs of an idealized gap must sit in the same congruence class with some prime pair not considered under the regime of gaps, establishing equidistribution reduces to equidistribution among the classes of all prime pairs in the bounded region considered in the proof of Lemma \ref{starter}. In light of this, we show in the mean time that each of the classes must contain at least two elements, if we allow for sufficiently large length of the boundaries.
\end{remark}
\bigskip

\begin{proposition}
Let \begin{align}\mathcal{X}_q=\left\{(p_1,p_2):(p_1,p_2)\in (a,b)\pmod q,~\gcd(p_1,q)=\gcd(p_2,q)=1\right \},\nonumber
\end{align}for $q\leq x$ then $\#\mathcal{X}_q\geq 2$ for some congruence class $(a,b)\pmod q$.
\end{proposition}

\begin{proof}
By Lemma \ref{starter}, the number of prime pairs within the bounded region with boundaries containing primes $p\leq x$ is given by \begin{align}\frac{\pi(x)(\pi(x)-1)}{2}>x-1\geq q\nonumber
\end{align}for $x$ sufficiently large. Thus the result follows from the pigeon-hole principle.
\end{proof}
\bigskip

In fact, we could have done better than just counting the number of distinct prime pairs, by counting the number of prime pairs of an idealized gap. It is important to notice that counting prime pairs of an idealized gap is not a trivial game compared to just counting prime pairs, since counting the number of prime pairs of the forms $(p_i,p_{i+z})$ is equivalent to understanding the correlation \begin{align}\sum \limits_{n\leq x}\theta(n)\theta(n+z)\nonumber
\end{align}where \begin{align}\theta(n):=\begin{cases}\log p \quad \mathrm{if}\quad n=p\\0 \quad \mathrm{otherwise}\end{cases}\nonumber
\end{align}for a fixed $z$. Sieve theory is a tool perfectly suited for this, which we shall use together with another method to obtain a rough estimate for the number of such prime pairs.  In \cite{Agama2019TheAM}, the author developed the area method. As an immediate consequence, we establish a lower bound for the number of prime pairs of the form $(p_i,p_{i+z})$ for a fixed $z$.
\bigskip

\begin{lemma}\label{cheb}
The estimate holds\begin{align}\sum \limits_{n\leq x}\theta(n)=(1+o(1))x.\nonumber
\end{align}
\end{lemma}
\bigskip
 
Next we present a flavour of the Area method. This method has already been developed in \cite{Agama2019TheAM}. As is useful for counting the number of prime pairs of an arbitrary gap, we do not hesitate to reproduce it here.

\begin{theorem}\label{identity 1}
Let $\{r_j\}_{j=1}^{n}$ and $\{h_j\}_{j=1}^{n}$ be any sequence of real numbers, and let $r$ and $h$ be any real numbers satisfying $\sum \limits_{j=1}^{n}r_j=r~\mathrm{and}~\sum \limits_{j=1}^{n}h_j=h$,  and \begin{align}(r^2+h^2)^{1/2}=\sum \limits_{j=1}^{n}(r^2_j+h^2_j)^{1/2},\nonumber
\end{align}then \begin{align}\sum \limits_{j=2}^{n}r_jh_j=\sum \limits_{j=2}^{n}h_j\bigg(\sum \limits_{i=1}^{j}r_i+\sum \limits_{i=1}^{j-1}r_i\bigg)-2\sum \limits_{j=1}^{n-1}r_j\sum \limits_{k=1}^{n-j}h_{j+k}.\nonumber
\end{align}
\end{theorem}

\begin{proof}
Consider a right angled triangle, say $\Delta ABC$ in a plane, with height $h$ and base $r$. Next, let us partition the height of the triangle into $n$ parts, not neccessarily equal. Now, we link those partitions along the height to the hypothenus, with the aid of a parallel line. At the point of contact of each line to the hypothenus, we drop down a vertical line to the next line connecting the last point  of the previous partition, thereby forming another right-angled triangle, say $\Delta A_1B_1C_1$ with base and height $r_1$ and $h_1$ respectively. We remark that this triangle is covered by the triangle $\Delta ABC$, with hypothenus constituting a proportion of the hypothenus of triangle $\Delta ABC$. We continue this process until we obtain $n$ right-angled triangles $\Delta A_jB_j C_j$, each with base and height $r_j$ and $h_j$ for $j=1,2,\ldots n$. This construction satisfies \begin{align}h=\sum \limits_{j=1}^{n}h_j~ \mathrm{and}~r=\sum \limits_{j=1}^{n}r_j \nonumber
\end{align}and \begin{align}(r^2+h^2)^{1/2}=\sum \limits_{j=1}^{n}(r^2_j+h^2_j)^{1/2}.\nonumber
\end{align}Now, let us deform the original triangle $\Delta ABC$ by removing the smaller triangles $\Delta A_jB_jC_j$ for $j=1,2,\ldots n$. Essentially we are left with rectangles and squares piled on each other with each end poking out a bit further than the one just above, and we observe that the total area of this portrait is given by the relation \begin{align}\mathcal{A}_1&=r_1h_2+(r_1+r_2)h_3+\cdots (r_1+r_2+\cdots +r_{n-2})h_{n-1}+(r_1+r_2+\cdots +r_{n-1})h_n\nonumber \\&=r_1(h_2+h_3+\cdots h_n)+r_2(h_3+h_4+\cdots +h_n)+\cdots +r_{n-2}(h_{n-1}+h_n)+r_{n-1}h_n\nonumber \\&=\sum \limits_{j=1}^{n-1}r_j\sum \limits_{k=1}^{n-j}h_{j+k}.\nonumber
\end{align} On the other hand, we observe that the area of this portrait is the same as the difference of the area of triangle $\Delta ABC$ and the sum of the areas of triangles $\Delta A_jB_jC_j$ for $j=1,2, \ldots,n$. That is \begin{align}\mathcal{A}_1=\frac{1}{2}rh-\frac{1}{2}\sum \limits_{j=1}^{n}r_jh_j.\nonumber
\end{align}This completes the first part of the argument. For the second part,  along the hypothenus, let us construct small pieces of triangle, each of base and height $(r_i, h_i)$ $(i=1,2 \ldots, n)$ so that the trapezoid and the one triangle formed by partitioning becomes rectangles and squares. We observe also that this construction satisfies the relation \begin{align}(r^2+h^2)^{1/2}=\sum \limits_{i=1}^{n}(r^2_i+h^2_i)^{1/2},\nonumber
\end{align}Now, we compute the area of the  triangle in two different ways. By direct strategy, we have that the area of the triangle, denoted $\mathcal{A}$, is given by \begin{align}\mathcal{A}=1/2\bigg(\sum \limits_{i=1}^{n}r_{i}\bigg)\bigg(\sum \limits_{i=1}^{n}h_{i}\bigg).\nonumber
\end{align} On the other hand, we compute the area of the triangle by computing the area of each trapezium and the one remaining triangle and sum them together. That is, \begin{align}\mathcal{A}&=h_n/2\bigg(\sum \limits_{i=1}^{n}r_{i}+\sum \limits_{i=1}^{n-1}r_{i}\bigg)+h_{n-1}/2\bigg(\sum \limits_{i=1}^{n-1}r_{i}+\sum \limits_{i=1}^{n-2}r_{i}\bigg)+ \cdots +1/2r_1h_1.\nonumber
\end{align} By comparing the area of the second argument, and linking this to the first argument, the result follows immediately.
\end{proof}

\begin{corollary}\label{decomposition}
Let $f:\mathbb{N}\longrightarrow \mathbb{C}$, then we have the decomposition \begin{align}\sum \limits_{n\leq x-1} \sum \limits_{j\leq x-n} f(n)f(n+j)&=\sum \limits_{2\leq n\leq x}f(n)\sum \limits_{m\leq n-1}f(m).\nonumber
\end{align}
\end{corollary}

\begin{proof}
Let us take $f(j)=r_j=h_j$ in Theorem \ref{identity 1}, then we denote $\mathcal{G}$ the partial sums \begin{align}\mathcal{G}=\sum \limits_{j=1}^{n}f(j)\nonumber
\end{align}and we notice that \begin{align}\sqrt{(\mathcal{G}^2+\mathcal{G}^2)}&=\sum \limits_{j=1}^{n}\sqrt{(f(j)^2+f(j)^2}\nonumber \\&=\sqrt{2}\sum \limits_{j=1}^{n}f(j).\nonumber
\end{align}Since $\sqrt{(\mathcal{G}^2+\mathcal{G}^2)}=\sqrt{2}\sum \limits_{j=1}^{n}f(j)$ our choice of sequence is  valid and, therefore the decomposition is valid for any arithmetic function.
\end{proof}
\bigskip

Corollary \ref{decomposition} provides us with an important identity that allows us to decompose any double correlation of an arithmetic function into a weighted sum averaged over a certain range. This tool gives the wriggle room to estimate an average correlation and to a larger extent a single correlation of an arithmetic function. This identity is the main and one of the important ingredient in this paper, and we will lean on it in many ways as time goes by.

\begin{lemma}(Area method)\label{key 2}
Let  $f:\mathbb{N}\longrightarrow \mathbb{R}^{+}$, a real-valued function. If \begin{align}\sum \limits_{n\leq x}f(n)f(n+l_0)>0\nonumber
\end{align}then there exist some constant $\mathcal{C}:=\mathcal{C}(l_0)>0$ such that \begin{align}\sum \limits_{n\leq x}f(n)f(n+l_0)\geq  \frac{1}{\mathcal{C}(l_0)x}\sum \limits_{2\leq n\leq x}f(n)\sum \limits_{m\leq n-1}f(m).\nonumber
\end{align}
\end{lemma}

\begin{proof}
By Corollary \ref{decomposition}, we obtain the identity by taking $f(j)=r_j=h_j$ \begin{align}\sum \limits_{n\leq x-1} \sum \limits_{j\leq x-n} f(n)f(n+j)&=\sum \limits_{2\leq n\leq x}f(n)\sum \limits_{m\leq n-1}f(m).\nonumber
\end{align}It follows that \begin{align}\sum \limits_{n\leq x-1} \sum \limits_{j\leq x-n} f(n)f(n+j)&\leq \sum \limits_{n\leq x} \sum \limits_{j\leq x} f(n)f(n+j)\nonumber \\&=\sum \limits_{n\leq x}f(n)f(n+1)+\sum \limits_{n\leq x}f(n)f(n+2)\nonumber \\&+\cdots \sum \limits_{n\leq x}f(n)f(n+l_0)+\cdots \sum \limits_{n\leq x}f(n)f(n+x)\nonumber \\&\leq  |\mathcal{M}(l_0)|\sum \limits_{n\leq x}f(n)f(n+l_0)\nonumber \\&+|\mathcal{N}(l_0)|\sum \limits_{n\leq x}f(n)f(n+l_0)\nonumber \\& +\cdots +\sum \limits_{n\leq x}f(n)f(n+l_0)+\cdots +|\mathcal{R}(l_0)|\sum \limits_{n\leq x}f(n)f(n+l_0)\nonumber \\&=\bigg(|\mathcal{M}(l_0)|+|\mathcal{N}(l_0)|+\cdots +1\nonumber \\&+\cdots +|\mathcal{R}(l_0)|\bigg)\sum \limits_{n\leq x}f(n)f(n+l_0)\nonumber \\&\leq \mathcal{C}(l_0)x\sum \limits_{n<x}f(n)f(n+l_0).\nonumber
\end{align}where $\mathrm{max}\{|\mathcal{M}(l_0)|, |\mathcal{N}(l_0)|, \ldots , |\mathcal{R}(l_0)|\}=\mathcal{C}(l_0)$. By inverting this inequality, the result follows immediately.
\end{proof}
\bigskip

\begin{theorem}
Let $z$ be fixed and let $\pi_z(x):=\# \{(p_i,p_{i+z}):p_i\leq x,~p_{i+z}\leq x\}$, then \begin{align}\pi_{z}(x)\geq (1+o(1))\frac{x}{2\mathcal{C}(z)\log^2x}\nonumber
\end{align}for some $\mathcal{C}(z)>0$.
\end{theorem}

\begin{proof}
Let $z$ be fixed and choose $f:=\theta$ in Lemma \ref{key 2}, then by Lemma \ref{cheb} we obtain the lower bound \begin{align}\sum \limits_{n\leq x}\theta(n)\theta(n+z)\geq (1+o(1))\frac{x}{2\mathcal{C}(z)}\nonumber
\end{align}for some $\mathcal{C}(z)>0$.
It follows that we can write\begin{align}\sum \limits_{n\leq x}\theta(n)\theta(n+z)=\sum \limits_{p,p+z\leq x}(\log p)(\log (p+z)).\nonumber
\end{align}The claimed lower bound follows by partial summation.
\end{proof}
\bigskip

This result in and of itself is of independent interest, for it solves the old conjecture of De Poligna about the infinitude of varieties of expressing even numbers as a difference of two prime numbers. The twin prime conjecture follows from this as a particular case by taking $z=2$. This result has already been exposed in a separate  paper, and so the focus here is not to elucidate on the breakthrough but we deem it necessary in this work, since it allows us to obtain the right order of growth of prime pairs with an idealized gap when combined with various upper bound derived from the methods of sieve theory.
\bigskip

\begin{theorem}(Brun)\label{brun}
Let $\alpha \in \mathbb{Z}$ for $\alpha\neq 0$, we have \begin{align}\# \{p\leq x:|p+\alpha| \quad \mathrm{is~prime}\}\leq \frac{cx}{\log^2x}\prod \limits_{p|\alpha}\bigg(1-\frac{1}{p}\bigg)^{-1}\nonumber
\end{align}for some absolute constant $c$.
\end{theorem}
\bigskip

Theorem \ref{brun} can be obtained from Brun's pure sieve \cite{cojocaru2006introduction}. This allows us to get control on the number of prime pairs of an ideal gap less than a fixed integer $x$. Combining this result with the lower bound gives us the right order of growth of the number of prime pairs of an idealized gap. We state this result in a formal manner as follows:
\bigskip

\begin{corollary}\label{brun 2}
Let $z$ be fixed and let $\pi_z(x):=\# \{(p_i,p_{i+z}):p_i\leq x,~p_{i+z}\leq x\}$, then\begin{align}\pi_z(x)\asymp_z \frac{x}{\log^2x}\nonumber
\end{align}where the implied constants depend on $z$.
\end{corollary}
\bigskip

It follows that establishing equidistribution among the prime pairs of an arbitrary gap, $z$ say, reduces to establishing the asymptotic
\begin{conjecture}
\begin{align}\pi_z(x;\mathcal{N}_q(a,b),q)\sim \frac{x\mathcal{D}(z)}{\phi(q)\log^2x}\nonumber
\end{align}for some constant $\mathcal{D}(z)>0$ and where $\phi(q)=\# \{\mathcal{N}_q(a_i,b_i)\}$. 
\end{conjecture}This is the theme of this paper, and every development thereof will be geared towards establishing this asymptotic, since we now know the right order of growth of prime pairs of an arbitrary gap.
\bigskip

Corollary \ref{brun 2} can also be expressed in terms of the Von mangoldt function by combining the area method with the upper bound
 
\begin{theorem}
The inequality is valid \begin{align}\sum \limits_{n\leq x}\Lambda(n)\Lambda(n+z)\ll \Im(z)x\nonumber
\end{align}where \begin{align}\Im(z)=2\Pi_2 \prod \limits_{p|z:p>2}\frac{p-1}{p-2}\nonumber 
\end{align}where $\Pi_2=\prod \limits_{p>2}\bigg(1-\frac{1}{(p-1)^2}\bigg)$ is the twin prime constant.
\end{theorem}
\bigskip

This eventually yields another reformulation of the problem in the form

\begin{conjecture}\label{form2}
\begin{align}\Psi_z(x;\mathcal{N}_q(a,b),q)\sim \frac{x\mathcal{D}(z)}{\phi(q)}\nonumber
\end{align}for some constant $\mathcal{D}(z)>0$ and where $\phi(q)=\# \{(a,b):(p_i,p_{i+z})\in \mathcal{N}_q(a,b)\}$. 
\end{conjecture}

\begin{remark}
Conjecture \ref{form2} is the form we will pursue in this paper.
\end{remark}

\begin{definition}[The $\omega$ function]
By the $\omega(q)$-function, we mean the function of the form \begin{align}\omega(q):=\# \{(a,b)\pmod q:(n_1,n_2)\equiv (a,b),~\gcd(n_1,q)=\gcd(n_2,q)=1\}.\nonumber
\end{align}
\end{definition}
\bigskip

\begin{theorem}\label{actual count}
Let $q\in \mathbb{N}$ , then we have \begin{align}\omega(q)=\frac{\phi(q)(\phi(q)-1)}{2}\nonumber
\end{align}where $\phi(q)$ is the euler-totient function.
\end{theorem}

\begin{proof}
Consider the positive Cartesian plane and slice it into two by a straight line so that the lower half contains finitely many points. Let $\omega(q)$ denotes the set of all pairs in this finite region whose entries are coprime to $q$. Next let us rotate this region clockwise and drop down a pependicular from the origin to the line. By letting $M$ and $N$ be the horizontal and the vertical boundary of the lower region, we set $|M|=|N|=q$ and consider coprime integers along the horizontal and the vertical boundary of the bounded region given by  $M_q=\{n\leq q: \gcd(n,q)=1\}$ and $N_q=\{n\leq q:\gcd(n,q)=1\}$. More explicitly we consider the sets \begin{align}M_q=\{n_1,n_2,\ldots, n_{\phi(q)}\}\nonumber
\end{align}and \begin{align}N_q=\{n_1,n_2,\ldots,n_{\phi(q)}\}.\nonumber
\end{align} Now we count the number of pairs obeying such a property in any half of the region, since by symetry the same count holds for the other half.  Let us count the number of classes of pairs in the lower half of the bounded region, which was a priori the upper half of the bounded region. To wit, we construct the set \begin{align}\mathcal{C}=\bigcup \limits_{j=1}^{\phi(q)-1}\{n_j\cdot n_{j+1},\ldots, n_{j}\cdot n_{\phi(q)}\}.\nonumber
\end{align}It follows that the number of classes of pairs  in ths bounded region is given by \begin{align}\# \mathcal{C} =\bigg|\bigcup \limits_{j=1}^{\phi(q)-1}\{n_j\cdot n_{j+1},\ldots, n_{j}\cdot n_{\phi(q)}\}\bigg|.\nonumber
\end{align}Since $\{n_i\cdot n_{i+1},\ldots,n_i\cdot n_{\phi(q)}\}\cap \{n_j\cdot n_{j+1},\ldots,n_{j}\cdot n_{\phi(q)}\}=\emptyset$ for $i\neq j$, it follows that \begin{align}\# \mathcal{C}&=\sum \limits_{j=1}^{\phi(q)-1}\# \{n_j\cdot n_{j+1},\ldots, n_j\cdot n_{\phi(q)}\}\nonumber \\&=\frac{\phi(q)(\phi(q)-1)}{2}\nonumber
\end{align}and the result follows immediately.
\end{proof}

\begin{corollary}\label{amountincoset}
Let $\gcd(a,q)=\gcd(b,q)=1$ and $\mathcal{N}=\{(n_1,n_2): (n_1,n_2)\simeq (a,b)\}$, then we have \begin{align}\# \mathcal{N}\leq \bigg \lceil \frac{\phi(q)-1}{2}\bigg \rceil.\nonumber
\end{align}
\end{corollary}

\begin{proof}
By Theorem \ref{actual count}, the omega function $\omega$ has the measure $\omega(q)=\frac{\phi(q)(\phi(q)-1)}{2}$. Appealing to Proposition \ref{coset}, the result follows immediately.
\end{proof}
\bigskip

Corollary \ref{amountincoset} puts a limit to the number of equivalent pairs that could possibly reside in each coset. We cannot in actuality opine on the exact count of these pairs let alone to compare their distribution. Nonetheless, we do not worry much about this lapse, since it does not interfere with our result.

\begin{lemma}\label{equidistribution}
Let $c_1>0$ be some constant.  For any $A>0$ there is some  $x(A)>0$ such that if $q<(\log x)^A$, then for $x\geq x(A)$ we have \begin{align}\Psi(x,\chi)=E_0(\chi)x+O\bigg(\frac{x}{e^{c_1\sqrt{\log x}}}\bigg)\nonumber
\end{align}where \begin{align}E_0(\chi)=\begin{cases}1 \quad \mathrm{if} \quad \chi=\chi_0\\0 \quad \mathrm{if} \quad \chi \neq \chi_0.\end{cases}\nonumber
\end{align}
\end{lemma}

\begin{proof}
For a proof see for instance \cite{May}.
\end{proof}

\begin{remark}
Lemma \ref{equidistribution} tells us that we will certainly have equidistribution of the primes among the primitive congruence classes if we allow for the modulus of progression  to grow polylogarithmic in size beyond a certain threshold controlled by the choice of the constant $A>0$. 
\end{remark}

\section{\textbf{The lattice characters $\kappa$ in the plane}}
In this section we study the Lattice characters. As is suggestive these characters are defined on integers lattice points. We study some of its properties from analytic and algebraic point of view. These tool will play a crucial role in trapping prime pairs of an idealized gap in a given primitive congruence class.
\bigskip

\begin{definition}(Multiplicative convolution)
Let $(m_1,n_1)$ and $(m_2,n_2)$ be any two integer lattice, then we set \begin{align}(m_1,n_1)\star (m_2,n_2)=(m_1m_2,n_1n_2).\nonumber
\end{align}
\end{definition}

\begin{remark}
Now we launch the lattice character which we deem indispensable for our next studies.
\end{remark}

\begin{definition}\label{convolution}
By the lattice character modulo $q$, we mean the arithmetic function  $\kappa:\bigcup_i\mathcal{N}(a_i,b_i) \longrightarrow \mathbb{C}$ such that $\kappa((m_1,n_1))=\chi(m_1n_1)$, where $\chi$ is the Dirichlet character.
\end{definition}

\subsection{Properties of the lattice character}

\begin{proposition}
The lattice character satisfies the following properties modulo $q$.
\begin{enumerate}
\item [(i)] $\kappa((m,n))=\kappa((n,m))$, (symmetric property).
\bigskip

\item [(ii)] $\kappa((m,n)+(q,q))=\kappa((m,n))$, (Periodicity).
\bigskip

\item [(iii)] $\kappa((m_1,n_1)\star (m_2,n_2))=\kappa((m_1,n_1))\kappa((m_2,n_2))$.
\bigskip

\item [(iv)] $\kappa((m,n))\neq 0$ if and only $\gcd(n,q)=\gcd(m,q)=1$.
\bigskip

\item [(v)] $\kappa((1,1))=1$.
\end{enumerate}
\end{proposition}

\begin{proof}
The first property is obvious. For the second, we notice that by the properties of the Dirichlet character modulo $q$ we have  $\kappa((m,n)+(q,q))=\kappa((m+q,n+q))=\chi((m+q)(n+q))=\chi(m+q)\chi(n+q)=\chi(m)\chi(n)=\kappa((m,n))$. The third property also follows trivially. Again, we notice that $\kappa((m,n))=\chi(mn)=\chi(m)\chi(n)\neq 0$ if and only if $\gcd(m,q)=\gcd(n,q)=1$. The last property also follows by noting that $\chi(1)=1$.
\end{proof}

\bigskip
The above definition of the lattice character is very natural in some sense. It purports lattice characters are multiplicative. This acclaimed structure of the character will enable us to use tools from analytic number theory in our study.  The profound reduction of the integer lattice and its connection to the classical Dirichlet character makes them very tractable to study. Indeed, we can certainly infer from the analytic and the algebraic properties of Dirichlet characters to the lattice character. These are just elementary properties of the lattice character and there are many more of these properties to study in the following sequel. 

\begin{lemma}\label{countcharacter}
There are $\phi(q)$ lattice characters modulo $q$.
\end{lemma}

\begin{proof}
Since there are $\phi(q)$ Dirichlet characters modulo $q$ (see \cite{hildebrand2005introduction}), It follows that there are $\phi(q)$ lattice characters modulo $q$ by virtue of definition \ref{convolution}.
\end{proof}

\begin{proposition}\label{properties}
Let $q$ be fixed and $\gcd(mn,q)=1$, then the following remain valid 
\begin{enumerate}
\item [(i)]\begin{align}\sum \limits_{(m,n)\in \bigcup_i\mathcal{N}(a_i,b_i)}\kappa((m,n))= \begin{cases}\phi(q) \quad \text{if} \quad \gcd(mn,q)=1\\0 \quad \text{otherwise}.\end{cases} \nonumber
\end{align}

\item [(ii)] \begin{align}\sum \limits_{\kappa}\kappa((m,n))= \begin{cases}\phi(q) \quad \text{if} \quad \chi=\chi_0\\0 \quad \text{otherwise}.\end{cases} \nonumber
\end{align}

\item [(iii)]\begin{align}\sum \limits_{(m,n)\in \bigcup_i\mathcal{N}(a_i,b_i)}\overline{\kappa_1((m,n))}\kappa_2((m,n))=\begin{cases}\phi(q) \quad \text{if} \quad \kappa_1=\kappa_2\\0 \quad \text{otherwise}.\end{cases}\nonumber
\end{align}

\item [(iv)] \begin{align}\sum \limits_{\kappa}\kappa((a,b))\overline{\kappa((m,n))}=\begin{cases}\phi(q) \quad \text{if} \quad (m,n)\equiv (a,b)\pmod q\\0 \quad \text{otherwise}.      \end{cases}\nonumber
\end{align}
\end{enumerate}
\end{proposition}

\begin{proof}
Proposition \ref{properties} follows from Definition \ref{convolution} and leveraging the properties of the Classical Dirichlet character. For details see for instance \cite{hildebrand2005introduction}.
\end{proof}
\bigskip

It is evident from Proposition \ref{properties} that property  (iv) can be recast as \begin{align}\sum \limits_{\kappa}\kappa((a,b))\overline{\kappa((m,n))}&=\sum \limits_{\chi}\chi(ab)\overline{\chi(mn)}\nonumber \\&=\sum \limits_{\chi}\chi(a)\chi(b)\overline{\chi(m)}\overline{\chi(n)}\nonumber \\&=\sum \limits_{\chi}\chi(a)\overline{\chi(m)}\chi(b)\overline{\chi(n)}\nonumber \\&=\sum \limits_{\chi}1\nonumber
\end{align}under the congruence condition $(a,b)\equiv (m,n)\pmod q$, which is equivalent to the two congruence conditions\begin{align}m\equiv a \pmod q \quad \mathrm{and} \quad n\equiv b\pmod q\nonumber
\end{align} with the underlying condition $\gcd(mn,q)=1$. By Lemma \ref{countcharacter}, property $(iv)$ is justified.

\begin{remark}
The last property of Theorem \ref{properties} is actually the main tool and certainly the beef of the problem at hand. We leverage this orthogonality principle to trap prime pairs in a given congruence class in the following sequel.
\end{remark}
\bigskip

\section{\textbf{Main theorem}}
In this section we give a proof of the main result of this paper. We assemble the tools we have laid down to establish the main result of this paper.

\begin{theorem}For some constant $c>0$ and for any $A>0$, there exist some $x(A)>0$  such that If $q\leq (\log x)^{A}$ then \begin{align}\Psi_z(x;\mathcal{N}_q(a,b),q)&=\frac{\Theta (z)}{2\phi(q)}x+O\bigg(\frac{x}{e^{c\sqrt{\log x}}}\bigg)\nonumber
\end{align}for $x\geq x(A)$.  In particular for $q\leq (\log x)^{A}$ for any $A>0$\begin{align}\Psi_z(x;\mathcal{N}_q(a,b),q)\sim \frac{x\mathcal{D}(z)}{2\phi(q)}\nonumber
\end{align}for some constant $\mathcal{D}(z)>0$ and where $\phi(q)=\# \{(a,b):(p_i,p_{i+z})\in \mathcal{N}_q(a,b)\}$. 
\end{theorem}

\begin{proof}
For $\gcd(ab,q)=1$, let us consider the following sums \begin{align}\Psi_z(x;\mathcal{N}_q(a,b),q)=\sum \limits_{\substack{n\leq x\\(n,n+z)\in \mathcal{N}_q(a,b)}}\Lambda(n)\Lambda(n+z)\nonumber
\end{align}and \begin{align}\Psi_z(x,\kappa)&=\sum \limits_{n\leq x}\Lambda(n)\Lambda(n+z)\kappa((n,n+z)).\nonumber
\end{align}By an application of the orthogonality principle in Theorem \ref{properties}, we can write\begin{align}1=\frac{1}{\phi(q)}\sum \limits_{\kappa}\kappa((n,n+z))\overline{\kappa((n,n+z))}.\label{beef}
\end{align}From \ref{beef}, we can write\begin{align}\sum \limits_{\substack{n\leq x\\(n,n+z)\in \mathcal{N}_q(a,b)}}\Lambda(n)\Lambda(n+z)&=\frac{1}{\phi(q)}\sum \limits_{\kappa}\overline{\kappa((a,b))}\sum \limits_{n\leq x}\Lambda(n)\Lambda(n+z)\kappa((n,n+z))\nonumber \\&=\frac{1}{\phi(q)}\sum \limits_{\chi}\overline{\chi(a)}\overline{\chi(b)}\Psi_z(x,\kappa)\nonumber \\&=\frac{1}{\phi(q)}\sum \limits_{\chi}\overline{\chi(a)\chi(b)}\sum \limits_{n\leq x}\Lambda\cdot \chi(n)\Lambda \cdot \chi(n+z)\nonumber \\&=\frac{\Theta (z)}{\phi(q)x}\sum \limits_{\chi}\overline{\chi(a)}\overline{\chi(b)}\sum \limits_{2\leq n\leq x}\Lambda\cdot \chi(n)\sum \limits_{m\leq n-1}\Lambda \cdot \chi(m)+O_{z,q}(1)\label{important}
\end{align}which is feasible by the area method and for some $\Theta(z)>0$. This reduces the problem to obtaining very good estimates for the sum \begin{align}\Psi(x,\chi)&=\sum \limits_{n\leq x}\Lambda(n)\chi(n).\nonumber
\end{align}By letting $q\leq (\log x)^{A}$ for any $A>0$ and appealing to Lemma \ref{equidistribution}, then we have \begin{align}\sum \limits_{2\leq n\leq x}\Lambda\cdot \chi(n)\sum \limits_{m\leq n-1}\Lambda \cdot \chi(m)&=E_0(\chi)\frac{x^2}{2}+O\bigg(\frac{x^2}{e^{c\sqrt{\log x}}}\bigg)\nonumber
\end{align}by an application of partial summation. By Plugging this estimate into \ref{important}, It follows that \begin{align}\sum \limits_{\substack{n\leq x\\(n,n+z)\in \mathcal{N}_q(a,b)}}\Lambda(n)\Lambda(n+z)&=\frac{\Theta (z)}{\phi(q)x}\sum \limits_{\chi}\overline{\chi(a)}\overline{\chi(b)}\bigg(E_0(\chi)\frac{x^2}{2}+O\bigg(\frac{x^2}{e^{c\sqrt{\log x}}}\bigg )\bigg)+O_{z,q}(1)\nonumber
\end{align}and the result follows immediately.
\end{proof}
\bigskip

The main result in this paper establishes an analogue of equidistribution of primes in arithmetic progression so long as we allow the modulus of our progression to grow polylogarithmic in size. Using a similar argument, It can be shown using the method, that under the same assumption of the main theorem \begin{theorem}\begin{align}\pi_z(x;\mathcal{N}_q(a,b),q)&=\frac{\Theta (z)}{2\phi(q)}\frac{x}{\log^2x}+O\bigg(\frac{x}{e^{c\sqrt{\log x}}}\bigg )\nonumber
\end{align}In particular \begin{align}\pi_z(x;\mathcal{N}_q(a,b),q)\sim \frac{\Theta(z)}{2\phi(q)}\frac{x}{\log^2x}\nonumber
\end{align}for $q\leq (\log x)^A$ for any $A>0$.
\end{theorem}

\footnote{
\par
.}%
.



\bibliographystyle{amsplain}

\end{document}